\documentclass[12pt]{amsart}

\usepackage{amsmath}
\usepackage{hyperref}
\usepackage[utf8]{inputenc} 
\usepackage[ruled,noline,noend]{algorithm2e}
\usepackage{mathtools}
\usepackage{enumerate}

\newtheorem{theorem}{Theorem}[section]

\newtheorem{lemma}[theorem]{Lemma}

\newtheorem{conjecture}[theorem]{Conjecture}

\newcommand{\bF}{\mathbb F}

        \DeclareMathOperator{\bw}{bw}

        \DeclareMathOperator{\rank}{rank}
        \DeclareMathOperator{\iso}{ISO}

        \newcommand{\del}{\backslash}
        \newcommand{\con}{/}

\begin{document}
\sloppy
	
\title[Matroid fragility]{Matroid fragility and relaxations of
circuit hyperplanes}
	
\author{Jim Geelen}
\author{Florian Hoersch}
\address{Department of Combinatorics and Optimization, University of Waterloo, Waterloo, Canada} 
\thanks{This research was partially supported by grants from the
Office of Naval Research [N00014-10-1-0851] and NSERC [203110-2011].}
	
\subjclass{05B35}
\keywords{Rota's Conjecture, matroid fragility, circuit-hyperplane relaxation}
\date{\today}
	
\begin{abstract}
We relate two conjectures that play a central role in the reported proof of Rota's Conjecture.
Let $\bF$ be a finite field. The first conjecture states that: the branch-width of any $\bF$-representable 
$N$-fragile matroid is bounded by a function depending only upon $\bF$ and $N$.
The second conjecture states that: if a matroid $M_2$ is obtained from a matroid $M_1$ by relaxing
a circuit-hyperplane and both $M_1$ and $M_2$ are $\bF$-representable, then the branch-width of
$M_1$ is bounded by a function depending only upon $\bF$. Our main result is that the second conjecture 
implies the first.
\end{abstract}
\maketitle
	
\section{Introduction}
	
The purpose of this paper is to relate two concepts,
$N$-fragile matroids and circuit-hyperplane relaxations,
which both play a central role in the reported proof
of Rota's Conjecture~[\ref{ggw}].
 
A matroid $M$ is {\em $N$-fragile}
if $N$ is a minor of $M$, but there is no element $e\in E(M)$
such that $N$ is a minor of both $M\del e$ and $M\con e$ or, equivalently, there is
a unique partition $(C,D)$ of $E(M)-E(N)$ such that $N=M\con C\del D$.
Note that here we want $N$, itself, as a minor, not just an
isomorphic copy of $N$.
		
For a finite field $\bF$ of order $q$, we let $\bF^k$ denote
an extension field of $\bF$ of order $q^k$. We prove the
following result.
\begin{theorem}\label{main}
Let $\bF$ be a finite field, let $N$ be a matroid with $k$ elements,
let $B$ be a basis of $N$, and
let $M$ be an $\bF$-representable $N$-fragile matroid.
Then there exist
$\bF^{2k^2}$-representable matroids $M_1$ and $M_2$ on the same
ground set and elements $c,d\in E(M_1)$ such that
$M_2$ is obtained from $M_1$ by relaxing a circuit-hyperplane
and $M\con B\del (E(N)-B) = M_1\con c\del d$.
\end{theorem}

The proof of Rota's Conjecture relies on the reported proofs of the 
following two conjectures by Geelen, Gerards, and Whittle.
\begin{conjecture}\label{conj1}
Let $\bF$ be a finite field and let $N$ be a matroid.
Then the branch-width of any $\bF$-representable $N$-fragile matroid
is bounded by a constant depending only upon $|\bF|$ and $|N|$. 
\end{conjecture}
For the definition of branch-width see Oxley~[\ref{oxley}].
For this paper it suffices to know that branch-width is a parameter associated with a matroid $M$, which we denote here by $\bw(M)$, and 
that for any minor $N$ of $M$ we have
$$ \bw(M) - (|E(M)|-|E(N)|)\le \bw(N) \le \bw(M).$$
\begin{conjecture}\label{conj2}
Let $H$ be a circuit-hyperplane in a matroid $M_1$ 
and let $M_2$ be the matroid obtained by relaxing $H$.
If $M_1$ and $M_2$ are both representable over a finite field $\bF$,
then the branch-width of $M_1$ is bounded by a constant
depending only upon $|\bF|$.
\end{conjecture}
Theorem~\ref{main} shows that Conjecture~\ref{conj2} implies
Conjecture~\ref{conj1}.

Our proof of Theorem~\ref{main} is via  a sequence of results on matrices,
but those results have interesting consequences for matroids, which
we state below.

We call a matroid {\em isolated} if each of its components 
has only one element.  Thus an isolated matroid
consists only of loops and coloops; the set
of coloops is the unique basis. The isolated 
matroid on ground set $E$ with basis $B$ is denoted
$\iso(B,E)$. For integers $r$ and $n$ with $0\le r\le n$
we denote $\iso(\{1,\ldots,r\},\{1,\ldots,n\})$ by $\iso(r,n)$.

The following result shows that, in order to
prove Theorem~\ref{main}, it suffices to
consider the case that $N$ is an isolated matroid.
\begin{theorem}\label{red1}
Let $\bF$ be a finite field, 
let $B$ be a basis of a matroid $N$,
and let $M$ be an $\bF$-representable $N$-fragile matroid.
Then there exists an $\bF$-representable $\iso(B,E(N))$-fragile
matroid $M'$ such that $E(M')=E(M)$ and $M'\con B=M\con B$.
\end{theorem}

The following result shows that, in order to
prove Theorem~\ref{main}, it suffices to
consider the case that $N=\iso(1,2)$.
\begin{theorem}\label{red2}
Let $\bF$ be a finite field, let $X_1$ and $X_2$ be disjoint finite sets with $|X_1\cup X_2|=k$, 
let $M$ be an $\bF$-representable $\iso(X_1,X_1\cup X_2)$-fragile matroid,
and let $c$ and $d$ be distinct elements not in $M$.
Then there exists an $\bF^{k^2}$-representable $\iso(\{c\},\{c,d\})$-fragile
matroid $M'$ such that $E(M')=E(M)-(X_1\cup X_2)\cup\{c,d\}$ and
$M'\con c\del d=M\con X_1\del X_2$.
\end{theorem}

The final result shows that an $\bF$-representable
$\iso(1,2)$-fragile matroid has a 
circuit-hyperplane whose relaxation results in an
$\bF^2$-representable matroid.
\begin{theorem}\label{red3}
Let $N=\iso(\{c\},\{c,d\})$ where $c\neq d$, let $M$ be an $N$-fragile matroid representable over a finite field $\bF$,
and let $C$ and $D$ be disjoint subsets of $E(M)$ such that $N= M\con C\del D$.
Then $C\cup\{d\}$ is a circuit-hyperplane of $M$ and the matroid obtained from $M$
by relaxing $C\cup\{d\}$ is $\bF^2$-representable.
\end{theorem}

Observe that Theorem~\ref{main} is an immediate consequence of Theorems~\ref{red1},~\ref{red2},
and~\ref{red3}.

We assume that the reader is familiar with elementary matroid theory; we use the
terminology and notation of Oxley~[\ref{oxley}].

\section{Fragile matrices}

In this section we will give a matrix interpretation for minor-fragility in representable matroids.
Towards this end, we develop convenient terminology for viewing  a representable
matroid with respect to a fixed basis.

For a basis $B$ of a matroid $M$ and a set $X\subseteq E(M)$ we denote the
minor $M \con (B-X) \del (E(M)-(B\cup X))$ of $M$ by $M[X,B]$.
The following result is routine and well-known.
\begin{lemma}\label{visible}
If $N$ is a minor of a matroid $M$, then there is a basis $B$ of $M$  such that
$N = M[E(N),B]$.
\end{lemma}

If $B$ is a basis of a matroid $M$ and
$N = M[E(N),B]$, then we say that $B$ {\em displays} $N$.

When we refer to a matrix $A\in\bF^{S_1\times S_2}$ we are implicitly defining $\bF$ to be a field
and $S_1$ and $S_2$ to be finite sets. Let $A\in\bF^{S_1\times S_2}$ be a matrix where 
$S_1$ and $S_2$ are disjoint. We let $[I,A]$ denote the matrix obtained from $A$ 
by appending an $S_1\times S_1$ identity matrix; thus $[I,A] \in\bF^{S_1\times(S_1\cup S_2)}$.
For $X\subseteq S_1\cup S_2$, we let $A[X]$ denote the submatrix $A[X\cap S_1,X\cap S_2]$.

If $B$ is a basis of an $\bF$-representable matroid $M$, then there is a matrix $A\in \bF^{B\times E(M)-B}$
such that $M=M([I,A])$; we call $A$ a {\em standard representation} with respect to $B$.
Note that, if $N$ is a minor of $M$ displayed by $B$ and $A$ is a standard representation 
of $M$ with respect to $B$, then $A[E(N)]$ is a standard representation of $N$ with respect to the 
basis $B\cap E(N)$.

For a finite set $X$, a matrix $A\in\bF^{S_1\times S_2}$ is 
called {\em $X$-fragile} if
\begin{itemize}
\item $S_1$ and $S_2$ are disjoint,
\item $X\subseteq S_1\cup S_2$,
\item $A[X]=0$, and 
\item for each nonempty subset $Y$ of $(S_1\cup S_2)-X$, we have
$\rank(A[X\cup Y])>\rank(A[Y])$.
\end{itemize}

Note that, if $A\in \bF^{S_1\times S_2}$ is an $X$-fragile matrix, then $M([I,A[X]]) = \iso(X\cap S_1,X)$.

The following result provides us with a matrix interpretation of minor-fragility for representable matroids.
\begin{lemma}\label{matrix1}
Let $N$ be a matroid, let $M$ be an $\bF$-representable $N$-fragile matroid,
let $B$ be a basis of $M$ that displays $N$, and let $A$ be a standard representation
of $M$ with respect to $B$. If $A'$ is the matrix obtained from $A$ by replacing each entry in the submatrix $A[E(N)]$
with $0$, then $A'$ is $E(N)$-fragile.
\end{lemma} 

\begin{proof} Let $X=E(N)$.
Suppose that $A'$ is not $X$-fragile. Then there is a non-empty set $Y\subseteq E(M)-X$ such that
$\rank(A'[X\cup Y])=\rank(A'[Y])$.
By removing the other elements, we may assume that $E(M)=X\cup Y$.
Let $C=B\cap Y$, $D=Y-B$, and let $B_N = B\cap E(N)$.
Observe that $\rank(A')=\rank(A[C,D])$ by the choice of $Y$.
We will obtain a contradiction to
the fact that $M$ is $N$-fragile by showing that $N=M\con D\del C$.

We start by constructing an isomorphic copy $A_0$ of $A'[B,X-B]$ by relabelling
the columns so that the indices form a set $Z$ disjoint from $E(N)$.
Now let $A_1= [A,A_0]$ and $M_1=M([I,A_1])$.

We claim that:
\begin{itemize}
\item[(i)] $N=(M_1\con Z)|X$, and
\item[(ii)] $B_N$ is independent in $M_1\con (D\cup Z)$, and
\item[(iii)] $Z$ is a set of loops in $M_1\con D$.
\end{itemize}

Note that $Z$ is a set of loops in $M_1\con C$ and $N$ is a minor of $M_1\con C$,
so $M_1\con Z$ contains $N$ as a minor.  To show that $N$ is a restriction of 
$M_1\con Z$ it suffices to show that $B_N$ spans $E(N)$ in $M_1\con Z$, or, equivalently,
that $B_N\cup Z$ spans $E(N)$ in $M_1$, which is clear from the construction.
This proves  $(i)$.

Note that $r_{M_1}(B_N\cup D\cup Z) = |B_N|+ \rank(A_1[C,D\cup Z]) =
|B_N|+ \rank(A'[C,D\cup X]) = |B_N|+ \rank(A[C,D])
= |B_N|+ \rank(A[B,D])$, since $\rank(A')=\rank(A[C,D])$.
Therefore $B_N$ is independent in in $M_1\con (D\cup Z)$, proving $(ii)$.

Now $(iii)$ follows directly from the fact that $\rank(A')=\rank(A[C,D])$.

By $(iii)$, we have  $M\con D =(M_1\con D)\del Z = (M_1\con D) \con Z$.
By $(i)$, $N$ is a restriction of $M_1\con Z$. By $(ii)$,
the sets $B_N$ and $D$ are skew in $M_1\con Z$ (that is, $r_{M_1\con Z}(B_N\cup D)
=r_{M_1\con Z}(B_N) + r_{M_1\con Z}(D)$), and hence 
$N$ is a restriction of $M_1\con(D\cup Z)$. However $M\con D = M_1\con(D\cup Z)$,
contradicting the fact that $M$ is $N$-fragile.
\end{proof}

The converse of Lemma~\ref{matrix1} is not true in general, but the following result is a weak converse,
and it implies Theorem~\ref{red1}.
\begin{lemma}\label{converse}
If  $A\in\bF^{S_1\times S_2}$ is an $X$-fragile matrix, where $X\subseteq S_1\cup S_2$,
then $M([I,A])$ is $\iso(X\cap S_1, X)$-fragile.
\end{lemma}

\begin{proof}
Let $M=M([I,A])$.
Note that $M[X,S_1] = \iso(X\cap S_1, X)$.
Let $C$ and $D$ be a partition  
of $E(M)-X$ such that $C\neq S_1-X$.
 We will prove that $M\con C\del D\neq \iso(X\cap S_1, X)$.
By contracting $C\cap S_1$ and deleting $D-S_1$ we may assume
that $D=S_1-X$ and that $C=S_2- X$.

Since $A$ is $X$-fragile, $\rank(A[D,C])< \rank(A)$. Now either
\begin{itemize}
\item[(i)] $\rank(A[D,C])< \rank(A[S_1,C])$, or
\item[(ii)] $\rank(A[S_1,C])< \rank(A)$.
\end{itemize}

In case $(i)$,  we have $r_{M\con C}(S_1\cap X) = 
r_M(C\cup (S_1\cap X)) - r_M(C) = |S_1\cap X| +\rank(A[D,C]) -\rank(A[S_1,C])
<|S_1\cap X|$. So $S_1\cap X$ is dependent in $M\con C$ and hence
$M\con C\del D\neq \iso(X\cap S_1, X)$, as required.

In case $(ii)$, we have $r_{M\con C}(X-S_1)
= r_M((X-S_1)\cup C) - r_M(C) = \rank(A) - \rank(A[S_1,C])>0$, so
$M\con C\del D\neq \iso(X\cap S_1, X)$, as required.
\end{proof}

\section{Reduction to $\iso(1,2)$-fragility}

The results in this section prove Theorem~\ref{red2}.

Let $F$ be a flat of a matroid $M$. We say that a matroid $M'$ is obtained 
by {\em adding an element $e$ freely to $F$ in $M$} if $M'$ is a single-element
extension by a new element $e$ in such a way that $F$ spans $e$ and that
each flat of $M'\del e$ that spans $e$ contains $F$. 
\begin{lemma}\label{iso:easy}
Let $M$ be an $\iso(X_1,X_1\cup X_2)$-fragile matroid, where
$X_1$ and $X_2$ are disjoint finite sets, and let
$M'$ be obtained from $M$ by adding a new element $d$ freely into the flat
spanned by $X_2$. Then  $M'\del X_2$ is $\iso(X_1,X_1\cup \{d\})$-fragile.
\end{lemma}

\begin{proof}
Let $(C,D)$ be a partition of $E(M)-(X_1\cup X_2)$.
It suffices to show that $M\con C\del D = \iso(X_1,X_1\cup X_2)$ if 
and only if $(M'\del X_2)\con C\del D = \iso(X_1,X_1\cup\{d\})$.
Note that $M'\con C\del D$ is obtained from $M\con C\del D$ by 
adding $d$ freely to the flat spanned by $X_2$.
If $M\con C\del D = \iso(X_1,X_1\cup X_2)$, then 
$M'\con C\del D = \iso(X_1,X_1\cup X_2\cup\{d\})$
and hence $(M'\del X_2)\con C\del D = \iso(X_1,X_1\cup \{d\})$.
Conversely, if $(M'\del X_2)\con C\del D = \iso(X_1,X_1\cup \{d\})$, then
$M'\con C\del D = \iso(X_1,X_1\cup X_2\cup\{d\})$
and hence $(M'\del \{d\})\con C\del D = \iso(X_1,X_1\cup X_2)$, as required.
\end{proof}

Note that, by Lemma~\ref{iso:easy}, we can reduce an $\iso(X_1,X_1\cup X_2)$-fragile matroid
to an $\iso(X_1,X_1\cup\{d\})$-fragile matroid. Repeating this in the dual we can further reduce
to an $\iso(\{c\},\{c,d\})$-fragile matroid.

We can add an element freely into a flat in a represented matroid by 
going to a sufficiently large extension field; this is both routine and well-known.
\begin{lemma}\label{extensions}
Let $A\in\bF^{S_1\times S_2}$, let $M=M(A)$, let 
$X$ be a $k$-element subset of $S_2$, and let 
$M'$ be the matroid obtained from $M$ by adding a new element $e$ 
freely into the flat spanned by $X$.
Then there is a vector $b\in (\bF^k)^{S_1}$ such that
$[A,b]$ is a representation of $M'$ over $\bF^k$.
\end{lemma}

\begin{proof}
Let $A_v$ denote the column of $A$ that is indexed by $v$.
The elements of the field $\bF^k$ form 
a vectorspace of dimension $k$ over $\bF$; let
$(\alpha_v\, :\, v\in X)$ be a basis 
of this vectorspace. Now let $b=\sum_{v\in X}\alpha_v A_v$
and let $M'=M([A,b])$. By construction, the new element $e$ of $M'$ is spanned by $X$.
It remains to show that each flat of $M'\del e$ that spans $e$ also spans $X$.
Consider an independent set $I\subseteq E(M)$ that does not span $X$ in $M$.
We may apply elementary row-operations over $\bF$ so that each column of $I$ contains 
exacly one non-zero entry. Let $R\subseteq S_1$ denote the set of rows containing non-zero entries
in $A[S_1,I]$. Since $I$ does not span $X$, there exists $i\in S_1-R$ such that
$A[\{i\},X]$ is not identically zero. However the entries of $A[\{i\},X]$ are all in $\bF$ and the values
$(\alpha_v\, :\, v\in X)$ are linearly independent over $\bF$, so $b_i = \sum_{v\in X} \alpha_v A_{i,v} \neq 0$.
Hence $I$ does not span $e$ in $M'$, as required.
\end{proof}

\section{Relaxing a circuit-hyperplane}

The following result implies Theorem~\ref{red3}.
\begin{lemma}\label{matrix3}
Let $\bF$ be a field and $\bF'$ be a field extension. 
Now let $A_1\in \bF^{S_1\times S_2}$
be a $\{c,d\}$-fragile matrix where $c\in S_1$ and $d\in S_2$ and let 
$A_2$ be obtained from $A_1$ by replacing the $(c,d)$-entry with 
an element in $\bF'-\bF$. Then $(S_1-\{c\})\cup \{d\}$ is a circuit-hyperplane
in $M([I,A_1])$ and $M([I,A_2])$ is the matroid obtained from $M([I,A_1])$
by relaxing $(S_1-\{c\})\cup \{d\}$.
\end{lemma}

\begin{proof}
Let $M_1=M([I,A_1])$, $M_2=M([I,A_2])$, and $H=(S_1-\{c\})\cup \{d\}$. We claim
that $H$ is a circuit of $M_1$; suppose otherwise.
Note that $S_1$ is a basis, so $S_1\cup \{d\}$ contains a unique circuit $C$.
Since $A_1$ is $\{c,d\}$-fragile, we have 
$A[\{c\},\{d\}]=0$, and hence $c\not\in C$.
Since $H$ is not a circuit, there exists $e\in S_1-\{c\}$
such that $e$ is a coloop of $M_1|(S_1\cup\{d\})$.
Then $(M_1|(S_1\cup\{d\}) )\del e = (M_1|(S_1\cup\{d\}))\con e$.
But then $M_1$ is not $\iso(\{c\},\{c,d\})$-fragile, contrary to
Lemma~\ref{converse}. Thus $H$ is a circuit as claimed. 

Note that $M_1^*=M([A_1^T,I])$ and that $A_1^T$ is
$\{c,d\}$-fragile. Then, by duality, $E(M_1)- H$ is a cocircuit and, hence, 
$H$ is a circuit-hyperplane.

To prove that $M_2$ is obtained from $M_1$
by relaxing $H$ it suffices to show, for each set $Z\subseteq S_1\cup S_2$, that
$\rank A_1[Z]\neq \rank A_2[Z]$ if and only if $Z=\{c,d\}$.
Note that $\rank A_1[\{c,d\}]\neq \rank A_2[\{c,d\}]$. Consider
a set $Z\subseteq S_1\cup S_2$ such that $\rank A_1[Z]\neq \rank A_2[Z]$.

\medskip

\noindent
{\bf Claim:}~~{\it We have $\rank A_1[Z]< \rank A_2[Z]$.}

\begin{proof}[Proof of claim.]
Suppose for a contradiction that $\rank A_1[Z]>\rank A_2[Z]$ and consider
a minimal subset $X\subseteq Z$ such that $\rank A_1[X]>\rank A_2[X]$.
Thus $A_1[X]$ is square and non-singular, $A_2[X]$ is singular, and $c,d\in X$.
Let $B(x)$ denote the matrix obtained from $A_1[X]$ by replacing the $(c,d)$-entry 
with a variable $x$ and let $p(x)= \det(B(x))$. Note that $p(x)= \alpha x+\beta$ where
$\alpha,\beta\in \bF$. Since $A_1[X]$ is non-singular, we have $p(0)\neq 0$. Therefore
$p(x)$ has at most one root and, since $\alpha,\beta\in \bF$, if $p(x)$ has a root, that root is in $\bF$.
However, this contradicts the fact that $A_2[X]$ is singular.
\end{proof}

By construction, $c,d\in Z$ and we may assume that $Z\ne \{c,d\}$.
Then, since $A_1$ is $\{c,d\}$-fragile,
\begin{eqnarray*}
 \rank A_1[Z-\{c,d\}] &\le & \rank A_1[Z]-1 \\
 &\le& \rank A_2[Z] -2\\
 &\le& \rank A_2[Z-\{c,d\}]  \\
 &=& \rank A_1[Z-\{c,d\}].
 \end{eqnarray*}
 Hence $\rank A_1[Z]=\rank A_1[Z-\{c,d\}]+1$ and $\rank A_2[Z] = \rank A_2[Z-\{c,d\}] +2$.
 This second equation implies that 
 $\rank A_2[Z-\{c\}] = \rank A_2[Z-\{c,d\}] +1$. Therefore
 $\rank A_1[Z-\{c\}] = \rank A_1[Z-\{c,d\}] +1$ and hence
 $\rank A_1[Z-\{c\}] = \rank A_1[Z]$.  Thus the row $c$ of $A_1[Z]$
 is a linear combination of the other rows.
 But then the row $c$ of $A_1[Z-\{d\}]$ is a linear combination of the other rows.
 So $\rank A_1[Z-\{d\}] = \rank A_1[Z-\{c,d\}]$ and, hence,
 $\rank A_2[Z-\{d\}] = \rank A_2[Z-\{c,d\}]$. However, this contradicts the fact that 
 $\rank A_2[Z] = \rank A_2[Z-\{c,d\}] +2$.
\end{proof}

	\section*{References}
	
	\newcounter{refs}
	
	\begin{list}{[\arabic{refs}]}%
		{\usecounter{refs}\setlength{\leftmargin}{10mm}\setlength{\itemsep}{0mm}}
\item \label{ggw}
J. Geelen, B. Gerards, G. Whittle, Solving Rota's Conjecture, Notices of the AMS 61 (2014), 736-743.

\item \label{oxley}
J. Oxley, \textit{Matroid Theory, second edition}, Oxford University Press, New York, (2011).
		\end{list}
	
\end{document}